\documentclass{article}
\usepackage[utf8]{inputenc}
    
\title{Odd Covers of Graphs}
\author{Calum Buchanan\thanks{{\tt calum.buchanan@uvm.edu}, University of Vermont, Burlington, VT} \and Alexander Clifton\thanks{{\tt aclift2@emory.edu}, Emory University, Atlanta, GA} \and Eric Culver\thanks{{\tt eric.culver@ucdenver.edu}, University of Colorado Denver, Denver, CO} \and Jiaxi Nie\thanks{{\tt jin019@ucsd.edu}, University of California San Diego, San Diego, CA} \and Jason O'Neill\thanks{{\tt jmoneill@ucsd.edu}, University of California San Diego, San Diego, CA} \and Puck Rombach\thanks{{\tt puck.rombach@uvm.edu}, University of Vermont, Burlington, VT} \and Mei Yin\thanks{{\tt mei.yin@du.edu}, University of Denver, Denver, CO}}
\date{\today}
\usepackage{graphicx}
\usepackage{amsmath, amssymb, amsthm, amsfonts}
\usepackage[dvipsnames]{xcolor}
\usepackage{color}
\usepackage{mathtools}
\usepackage{subcaption}
\usepackage{enumitem}
\usepackage{tikz}
\usepackage{hyperref}
\usepackage{mathrsfs}
\newcounter{problem}

\newtheorem{theorem}{Theorem}[section]

\newtheorem{lemma}[theorem]{Lemma}
\newtheorem{cor}[theorem]{Corollary}

\newtheorem{prob}[problem]{Problem}
\newtheorem{prop}[theorem]{Proposition}
\newtheorem{conj}[theorem]{Conjecture}

\theoremstyle{definition}
\newtheorem{definition}[theorem]{Definition}

\tikzstyle{vertex}=[draw,thick,fill=white,circle,inner sep=2pt]

\DeclareMathOperator{\bp}{bp}
\DeclareMathOperator{\bc}{bc}
\DeclareMathOperator{\spn}{span}

\newcommand{\blank}{\varepsilon}

\begin{document}

\maketitle

\begin{abstract}
Given a finite simple graph $G$, an {\em odd cover of $G$} is a collection of complete bipartite graphs, or bicliques, in which each edge of $G$ appears in an odd number of bicliques and each non-edge of $G$ appears in an even number of bicliques. We denote the minimum cardinality of an odd cover of $G$ by $b_2(G)$ and prove that $b_2(G)$ is bounded below by half of the rank over $\mathbb{F}_2$ of the adjacency matrix of $G$. We show that this lower bound is tight in the case when $G$ is a bipartite graph and almost tight when $G$ is an odd cycle. However, we also present an infinite family of graphs which shows that this lower bound can be arbitrarily far away from $b_2(G)$. 

Babai and Frankl (1992) proposed the ``odd cover problem," which in our language is equivalent to determining $b_2(K_n)$. Radhakrishnan, Sen, and Vishwanathan (2000) determined $b_2(K_n)$ for an infinite but density zero subset of positive integers $n$. In this paper, we determine $b_2(K_n)$ for a density $3/8$ subset of the positive integers.

\end{abstract}

\section{Introduction}
Let $G = (V, E)$ be a finite simple graph. An {\em odd cover of $G$} is a collection of bicliques, or complete bipartite graphs, on subsets of the vertex set $V$ in which two vertices are adjacent in an odd number of bicliques if and only if they are adjacent in $G$. We note that an odd cover always exists, since trivially, the collection of bicliques with partite sets $\{u\}$ and $\{v\}$ for each pair of adjacent vertices $u$ and $v$ in $G$ constitutes one such cover. It turns out that a collection of bicliques form an odd cover of $G$ if and only if the symmetric difference of their edge sets is $E(G)$. The minimum cardinality of an odd cover of $G$ is denoted $b_2(G)$.

The problem of determining $b_2(G)$ was posed by Neil de Beaudrap~\cite{76043}. Counterpart formulations of odd cover exist in the literature. In \cite{KAMINSKI20092747}, an operation termed a \emph{bipartite subgraph complementation} was introduced, which complements the edges between two disjoint subsets of vertices of the graph. Under this terminology, finding the minimum cardinality of an odd cover of a graph $G$ on $n$ vertices translates to identifying the minimum number of bipartite subgraph complementations needed to obtain $G$ from the empty graph on $n$ vertices. This is related to the problem studied in \cite{BPR} of finding the minimum number of {\em subgraph complementations}, the operation of complementing the edge set of an induced subgraph, needed to obtain $G$ from the empty graph.

There are two notions closely related to odd cover: biclique partition and biclique covering, both of which have been widely studied. A biclique partition of $G$ is a collection of edge-disjoint complete bipartite subgraphs of $G$ whose edges partition the edge set of $G$. The minimum cardinality of a biclique partition of $G$ is denoted $\bp(G)$. A biclique covering of $G$ is a collection of complete bipartite subgraphs of $G$ such that every edge of $G$ appears at least once. The minimum cardinality of a biclique covering of $G$ is denoted $\bc(G)$. It is not hard to see that every biclique partition is both a biclique covering and an odd cover, which implies that $\bc(G) \leq \bp(G)$ and $b_2(G) \leq \bp(G)$, but there are biclique coverings and odd covers which are not biclique partitions. Furthermore, there is an important distinction between each of these ideas and odd covers, as we can include non-edges of $G$ in an odd cover but not in a biclique partition or a biclique covering.

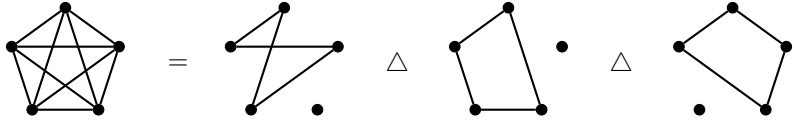
\begin{figure}
\begin{center}
\begin{tikzpicture}
[every node/.style={circle, draw=black!100, fill=black!100, inner sep=0pt, minimum size=4pt},
every edge/.style={draw, black!100, thick},
scale=.75]

\foreach \i in {1,...,5}
{
\node (\i) at (\i*360/5+18:1) {};
}
\draw[black!100,thick] (1) -- (2) -- (3) -- (4)-- (5)-- (1);
\draw[black!100,thick] (1) -- (3);
\draw[black!100,thick] (1) -- (4);
\draw[black!100,thick] (2) -- (4);
\draw[black!100,thick] (2) -- (5);
\draw[black!100,thick] (3) -- (5);

\node[draw=none,fill=none] at (2,0) {$=$};
\end{tikzpicture}
\quad
\begin{tikzpicture}
[every node/.style={circle, draw=black!100, fill=black!100, inner sep=0pt, minimum size=4pt},
every edge/.style={draw, black!100, thick},
scale=.75]

\foreach \i in {1,...,5}
{
\node (\i) at (\i*360/5+18:1) {};
}
\draw[black!100,thick] (1) -- (2);
\draw[black!100,thick] (1) -- (3);
\draw[black!100,thick] (2) -- (5);
\draw[black!100,thick] (3) -- (5);

\node[draw=none,fill=none] at (2,0) {$\triangle$};
\end{tikzpicture}
\quad
\begin{tikzpicture}
[every node/.style={circle, draw=black!100, fill=black!100, inner sep=0pt, minimum size=4pt},
every edge/.style={draw, black!100, thick},
scale=.75]

\foreach \i in {1,...,5}
{
\node (\i) at (\i*360/5+18:1) {};
}
\draw[black!100,thick] (1) -- (2) -- (3) -- (4) -- (1);

\node[draw=none,fill=none] at (2,0) {$\triangle$};
\end{tikzpicture}
\quad
\begin{tikzpicture}
[every node/.style={circle, draw=black!100, fill=black!100, inner sep=0pt, minimum size=4pt},
every edge/.style={draw, black!100, thick},
scale=.75]

\foreach \i in {1,...,5}
{
\node (\i) at (\i*360/5+18:1) {};
}
\draw[black!100,thick] (1) -- (2) -- (4) -- (5) -- (1);
\end{tikzpicture}
\end{center}

\caption{An odd cover of $K_5$. See Theorem \ref{clique}.}
\label{fig:K5}
\end{figure}

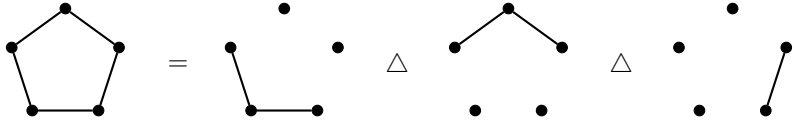
\begin{figure}
\begin{center}
\begin{tikzpicture}
[every node/.style={circle, draw=black!100, fill=black!100, inner sep=0pt, minimum size=4pt},
every edge/.style={draw, black!100, thick},
scale=.75]

\foreach \i in {1,...,5}
{
\node (\i) at (\i*360/5+18:1) {};
}
\draw[black!100,thick] (1) -- (2) -- (3) -- (4)-- (5)-- (1);

\node[draw=none,fill=none] at (2,0) {$=$};
\end{tikzpicture}
\quad
\begin{tikzpicture}
[every node/.style={circle, draw=black!100, fill=black!100, inner sep=0pt, minimum size=4pt},
every edge/.style={draw, black!100, thick},
scale=.75]

\foreach \i in {1,...,5}
{
\node (\i) at (\i*360/5+18:1) {};
}
\draw[black!100,thick] (2) -- (3) -- (4);

\node[draw=none,fill=none] at (2,0) {$\triangle$};
\end{tikzpicture}
\quad
\begin{tikzpicture}
[every node/.style={circle, draw=black!100, fill=black!100, inner sep=0pt, minimum size=4pt},
every edge/.style={draw, black!100, thick},
scale=.75]

\foreach \i in {1,...,5}
{
\node (\i) at (\i*360/5+18:1) {};
}
\draw[black!100,thick] (1) -- (2);
\draw[black!100,thick] (1) -- (5);

\node[draw=none,fill=none] at (2,0) {$\triangle$};
\end{tikzpicture}
\quad
\begin{tikzpicture}
[every node/.style={circle, draw=black!100, fill=black!100, inner sep=0pt, minimum size=4pt},
every edge/.style={draw, black!100, thick},
scale=.75]

\foreach \i in {1,...,5}
{
\node (\i) at (\i*360/5+18:1) {};
}
\draw[black!100,thick] (4) -- (5);
\end{tikzpicture}
\end{center}

\caption{An odd cover of $C_5$. See Theorem \ref{cycle-odd}.}
\label{fig:C5}
\end{figure}

\begin{figure}
\begin{center}
\begin{tikzpicture}
[every node/.style={circle, draw=black!100, fill=black!100, inner sep=0pt, minimum size=4pt},
every edge/.style={draw, black!100, thick},
scale=.75]

\foreach \i in {1,...,6}
{
\node (\i) at (\i*360/6:1) {};
}
\draw[black!100,thick] (1) -- (2) -- (3) -- (4)-- (5)-- (6) -- (1) ;

\node[draw=none,fill=none] at (2,0) {$=$};
\end{tikzpicture}
\quad
\begin{tikzpicture}
[every node/.style={circle, draw=black!100, fill=black!100, inner sep=0pt, minimum size=4pt},
every edge/.style={draw, black!100, thick},
scale=.75]

\foreach \i in {1,...,6}
{
\node (\i) at (\i*360/6:1) {};
}
\draw[black!100,thick] (1) -- (2) -- (3) -- (6) -- (1);

\node[draw=none,fill=none] at (2,0) {$\triangle$};
\end{tikzpicture}
\quad
\begin{tikzpicture}
[every node/.style={circle, draw=black!100, fill=black!100, inner sep=0pt, minimum size=4pt},
every edge/.style={draw, black!100, thick},
scale=.75]

\foreach \i in {1,...,6}
{
\node (\i) at (\i*360/6:1) {};
}
\draw[black!100,thick] (3) -- (4) -- (5) -- (6) -- (3);
\end{tikzpicture}
\end{center}

\caption{An odd cover of $C_6$. See Corollary \ref{cycle-even}.}
\label{fig:C6}
\end{figure}
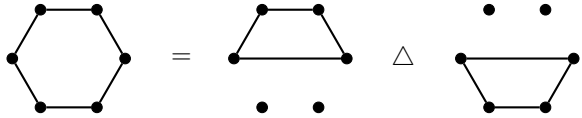

There are graphs for which $b_2(G)=\bp(G)=\bc(G)$, such as $C_5$ as depicted in Figure \ref{fig:C5}. But more interestingly, there are graphs for which $b_2(G)=\bc(G)<\bp(G)$, such as $K_5$ as depicted in Figure \ref{fig:K5}, as well as graphs for which $b_2(G)<\bc(G)=\bp(G)$, such as $C_6$ as depicted in Figure \ref{fig:C6}. Several other comparison possibilities exist between these three interconnected concepts.

For complete graphs, as we will show in Theorem \ref{clique}, $b_2(K_n)$ is bounded below by $\lceil n/2 \rceil$ and above by $\lceil n/2 \rceil +1$. The lower bound is established via evaluating the $2$-rank as in Proposition \ref{prop:ranklower} and the upper bound is established from constructing explicit odd covers. Further, there are some special $n$ values where the exact number for $b_2(K_n)$ is identified. (See the next paragraph for the significance of these special values.) It is known that asymptotically, $\bc(K_n)=\lceil \log_2 n \rceil$ and $\bp(K_n)=n-1$, thus $\bc(K_n) \lesssim b_2(K_n) \lesssim \bp(K_n)$. For odd cycles, as we will show in Theorem \ref{cycle-odd}, $b_2(C_n)=(n+1)/2$, and this coincides with the known asymptotic results for $\bc(C_n)$ and $\bp(C_n)$. The upper bound comes from a straightforward biclique partition, while the lower bound is achieved from performing rank analysis on the associated adjacency matrices. Contrarily, for even cycles, using Theorem \ref{thm:bipartite} and recognizing that an even cycle $C_n$ is the line graph for itself and thus has rank $n-2$ over  $\mathbb{F}_2$, we will show in Corollary \ref{cycle-even} that $b_2(C_n)=n/2-1$ whereas $\bc(C_n)=\bp(C_n)=n/2$ for $n\geq 6$ ($\bc(C_4)=\bp(C_4)=1$ is a special case).

For paths, using Theorem \ref{thm:bipartite} and recognizing that a path $P_n$ is the line graph for $P_{n+1}$ and thus has $2$-rank $n-1$ for $n$ odd and $n$ for $n$ even, we will show in Corollary \ref{path} that, independent of the parity of $n$, $b_2(P_n)=\lfloor n/2 \rfloor$, and this coincides with the known asymptotic results for $\bc(P_n)$ and $\bp(P_n)$.

Our result concerning complete graphs in particular may be regarded as an extension to the famous Graham-Pollak \cite{GP} theorem in algebraic graph theory, which states that $\bp(K_n)=n-1$. Instead of requiring that the bicliques be edge-disjoint, the odd cover problem only asks for each edge of the complete graph to be covered an odd number of times, and so as argued before, $b_2(K_n)$ could be asymptotically smaller than $\bp(K_n)$. This odd cover problem for complete graphs was considered earlier by Babai and Frankl \cite{babai2020}, who observed a lower bound of $\lfloor n/2 \rfloor$ for $b_2(K_n)$. Further progress was made by Radhakrishnan, Sen and Vishwanathan~\cite{radhakrishnan2000}, who determined $b_2(K_n) = \lceil n/2 \rceil$ when $n=2(q^2+q+1)$ and there exists a projective plane of order $q$ with $q=3 \mod{4}$, and also $b_2(K_{2n})=n$ when there exists an $n \times n$ Hadamard matrix. We note that the latter such $n$ are only known to comprise a density zero set of positive integers. (See \cite{Launey} for best known results on the density of integers for which there exists a Hadamard matrix.) By contrast, our explicit odd cover constructions show that $b_2(K_n)=\lceil n/2 \rceil$ when $n=8k$ or $8k \pm 1$ for some positive integer $k$, and hence solve the odd cover problem on complete graphs on a density $3/8$ portion of positive integers.

This paper is organized as follows. Section \ref{defnot} introduces definitions and notations. Section \ref{sec:oddcovers} presents a general lower bound on the minimum cardinality of an odd cover of a graph. We provide two algebraic perspectives on odd covers, both of which lead to this same lower bound, and one to a general upper bound. A family of graphs whose minimum odd cover cardinality is bounded away from this general lower bound is also identified. The later sections are concerned with finding the minimum odd cover cardinality for different classes of graphs and comparing it against the general lower bound. Section \ref{sec:bipartite} shows that the general lower bound is achieved for all bipartite graphs, while Section \ref{sec:odd cycles} shows that the general lower bound is off by one in the case of odd cycles. Section \ref{complete} shows that the difference between $b_2(K_n)$ and the general lower bound is within 2 for all $n$ and establishes the exact value of $b_2(K_n)$ when $n = 8k$ or $n = 8k \pm 1$ for some positive integer $k$. Section \ref{conclusion} gives some concluding remarks and open questions.

\section{Definitions and notations}
\label{defnot}
All graphs considered in this paper are finite and simple. The vertex set of a graph $G$ is denoted by $V(G)$ and the edge set by $E(G)$, or by $V$ and $E$, respectively, when the graph $G$ is evident from context. The number of vertices of $G$ is denoted by $|G|$ and the number of edges by $\|G\|$, or by $n$ and $m$ when $G$ is evident from context. The adjacency matrix of $G$ is denoted by $A(G)$. Complete graphs are denoted by $K_n$, paths by $P_n$, and cycles by $C_n$, where $n$ is the number of vertices in each case. Complete bipartite graphs, or bicliques, are denoted by $K_{a,b}$, where $a$ and $b$ are the numbers of vertices in each partite set. When it is evident from context, a biclique may be denoted by a pair $(X,Y)$, where $X$ and $Y$ are the partite sets of the biclique. Complete tripartite graphs, or tricliques, are similarly denoted by $K_{a,b,c}$, or by $(X,Y,Z)$. The disjoint union of graphs $G$ and $H$ is denoted by $G+H$, and the disjoint union of $k$ copies of $G$ by $kG$. The {\em symmetric difference} of two graphs $G$ and $H$ on the same vertex set $V$ is the graph $G \triangle H = (V, E(G) \triangle E(H))$, {\em i.e.}, whose edge set is the symmetric difference of $E(G)$ and $E(H)$. In this paper, we generalize this definition by considering symmetric differences of graphs on subsets of a vertex set $V$, defined in the same way. We denote by $N(v)$ the {\em open neighborhood} of a vertex $v$ of $G$, that is, $N(v) = \{u \mid uv \in E\}$, and by $N[v]$ the {\em closed neighborhood} of $v$, that is, $N[v] = N(v) \cup \{v\}$. Two vertices $u$ and $v$ of $G$ are said to be {\em twins} if $N(u) = N(v)$, and are said to be {\em adjacent twins} if $N[u] = N[v]$. The {\em degree} of $v$ is $|N(v)|$, denoted by $d(v)$. When the graph $G$ in question is not evident, we use the notations $N_G(v)$, $N_G[v]$, and $d_G(v)$, respectively. The induced subgraph of $G$ on the subset of vertices $V \setminus S$ is denoted by $G - S$, or by $G[V\setminus S]$, and the graph obtained by deleting a vertex $v$ or an edge $e$ from $G$ is denoted by $G - v$ or $G - e$, respectively.

\section{Odd covers of graphs}\label{sec:oddcovers}

We begin with a general lower bound on $b_2(G)$. We define the {\em $2$-rank} of a graph $G$, denoted $r_2(G)$, to be the rank of its adjacency matrix over $\mathbb{F}_2$, the finite field of order $2$. We will see that $2b_2(G) \geq r_2(G)$ in two ways, each of which yields a different algebraic perspective on odd covers of graphs.

\begin{prop}\label{prop:ranklower}
For any graph $G$,
\begin{equation*}\label{eq:ranklb}
 b_2(G) \geq \frac{r_2(G)}{2}   
\end{equation*}
\end{prop}

\begin{proof}
Let $\mathscr{B} = \{B_1, \ldots, B_k\}$ denote a minimum odd cover of $G$. To each biclique $B_i$ in $\mathscr{B}$, we add isolated vertices to obtain a graph $H_i$ on $V(G)$. It follows from the definition of an odd cover that $A(G)=\sum_{i=1}^k A(H_i) \pmod 2$. Since a biclique has rank at most 2 over any field, and since matrix rank is subadditive, we see that $r_2(G) \leq 2k = 2b_2(G)$. This completes the proof.
\end{proof}

There is another algebraic interpretation of an odd cover, from which we can obtain the same lower bound. We introduce it here as we will use a similar argument at the end of this section to derive an upper bound $b_2(G) \leq r_2(G)$ and introduce a family of graphs for which $b_2(G)$ is bounded away from $r_2(G) / 2$ asymptotically. Let $G$ be a graph with vertices enumerated $v_1, v_2, \ldots, v_n$, and let $\mathscr{B} = \{(X_1, Y_1), \ldots, (X_k,Y_k)\}$ be an odd cover of $G$. Assign to each vertex $v_i$ an incidence vector $(x_{i,1}, y_{i,1}, \ldots, x_{i,k}, y_{i,k}) \in \mathbb{F}_2^{2k}$ whose entries are in pairs $x_{i,j}, y_{i,j}$, where $x_{i,j}$ is $1$ if $v_i \in X_j$ and is $0$ otherwise, and $y_{i,j}$ is $1$ if $v_i \in Y_j$ and is $0$ otherwise. That is, the pair $(x_{i,j},y_{i,j})$ associated to the $i$th vertex and the $j$th biclique is $(1,0)$ if $v_i \in X_j$, is $(0,1)$ if $v_i \in Y_j$, and is $(0,0)$ otherwise. Let $M$ denote the $n \times 2k$ matrix whose rows are the incidence vectors for the vertices of $G$, and let $H_2 = \left(\begin{smallmatrix} 0 & 1 \\ 1 & 0 \end{smallmatrix} \right)$. We consider the $n\times n$ matrix $A = (a_{i,j})$ given by
\begin{equation}\label{eq:MoplusMt}
A = M (\oplus_1^k H_2) M^T,
\end{equation}
where matrix multiplication is taken over $\mathbb{F}_2$. Notice that the entry $a_{i,j}$ is the sum $\sum_{s = 1}^k (x_{is}y_{js} + x_{js}y_{is})$, of which the $s$th summand is 1 if $v_i$ and $v_j$ are in differing partite sets of the biclique $(X_s, Y_s)$, and is 0 otherwise. Thus, $a_{i,j}$ is 1 if and only if $v_i$ and $v_j$ are in differing partite sets of an odd number of bicliques. In other words,
\[A = A(G).\]
It is not hard to see that $r_2(M (\oplus_1^k H_2) M^T) \leq r_2(M) \leq 2k$, from which we again obtain the lower bound in Proposition~\ref{prop:ranklower}. We will revisit this idea at the end of this section.

We now introduce another interpretation of odd covers. We will examine a family of graphs $\{B_k \mid k \in \mathbb{N}\}$ such that each $B_k$ contains, as induced subgraphs, all of the twin-free graphs $G$ with $b_2(G) \leq k$. It is not hard to see that twin vertices do not affect $b_2(G)$. If $G$ is a graph with twin vertices $u$ and $v$, then we can obtain a minimum odd cover of $G$ from one of $G-u$ by including $u$ in every partite set in which $v$ appears. This fact also follows from Lemma~\ref{lem:nooddoddextension}.

\begin{definition}\label{defn:Bk}
    The graph $B_k$ has vertex set consisting of all strings of length $k$ with entries in $\{0,1, \blank\}$, where vertices $u$ and $v$ are adjacent if and only if the number of places where one contains $0$ and the other contains $1$ is odd. 
\end{definition}

\begin{prop}\label{prop:Bk}
Any twin-free graph $G$ with $b_2(G) \leq k$ is an induced subgraph of $B_k$.
\end{prop}

\begin{proof}
We consider a certain encoding of odd covers. Let $\{(X_1, Y_1), \ldots, (X_k, Y_k)\}$ be an odd cover of $G$, where each biclique $i$ is indicated by its two (possibly empty) partite sets $X_i$ and $Y_i$. We assign to each vertex $v$ a word from $\{0, 1, \blank\}^k$ by having the $i$th place contain a $0$ if $v \in X_i$, a $1$ if $v \in Y_i$, and $\blank$ if $v \not\in X_i \cup Y_i$.  Since, by definition, $X_i \cap Y_i = \emptyset$, this is well-defined. If two vertices, $v, u$ are assigned the same word by this scheme, then they are in the same partite sets of all the bicliques in the odd cover, and therefore are necessarily twins. Since $G$ is twin-free, all the assigned words are unique. Therefore, this assignment can be considered an injective mapping from the vertices of $G$ to the vertices of $B_k$. Also note that if vertices $v$ and $u$ are adjacent, then they are in opposite partite sets an odd number of times (by definition of an odd cover). Thus, the words assigned to $v$ and $u$ will have the property that one contains $0$ and the other contains $1$ in an odd number of places. Therefore, edges of $G$ are mapped to edges of $B_k$. In addition, if vertices $v$ and $u$ are not adjacent, then they are in opposite partite sets an even number of times, and so the words assigned to $v$ and $u$ will have the property that one contains $0$ and one contains $1$ in an even number of places. Therefore, non-edges of $G$ are mapped to non-edges of $B_k$. This means that there is an induced copy of $G$ in $B_k$. 
\end{proof}

We now have multiple ways to encode an odd cover. These notions are actually related; if we replace the entries $0$, $1$, and $\blank$ in the strings which denote the vertices of $B_k$ with pairs $(1,0)$, $(0,1)$, and $(0,0)$, respectively, to obtain a vector of length $2k$ for each vertex of $B_k$, then the matrix $M$ with rows given by these vectors has the property that $A(B_k) = M (\oplus_1^k H_2) M^T$, where $H_2 = \left(\begin{smallmatrix} 0 & 1 \\ 1 & 0 \end{smallmatrix}\right)$, as in equation~\eqref{eq:MoplusMt}.

There is an odd cover for any graph $G$, and thus there is always a matrix $M$ such that $A(G) = M (\oplus_1^s H_2) M^T$, where $s$ is the cardinality of the odd cover. It is not always the case that $b_2(G) = r_2(G) / 2$, so the matrix $M$ that we obtain from an odd cover does not always have $r_2(G)$ columns. However, a corollary to Theorem~2.6 in~\cite{friedland1991quadratic} provides us a partial converse.

\begin{prop}[\cite{friedland1991quadratic}]\label{prop:friedland}
Let $A$ be a symmetric $n \times n$ matrix over $\mathbb{F}_2$ of rank $r$ such that every diagonal entry of $A$ is 0. Then there exists an $n\times r$ matrix $M$ of rank $r$ such that
\[A = M (\oplus_1^k H_2) M^T,\]
where $r = 2k$ and $H_2 = \left(\begin{smallmatrix} 0 & 1 \\ 1 & 0 \end{smallmatrix}\right)$.
\end{prop}

It follows from Proposition~\ref{prop:friedland} that, for any $n$-vertex graph $G$ with 2-rank $r = 2k$, there exists an $n\times r$ matrix $M$ of rank $r$ satisfying $A(G) = M (\oplus_1^k H_2) M^T$. We have seen such a product in equation~\eqref{eq:MoplusMt}, and we observed that if the entries of the $i$th row of $M$ are paired $(x_{i,1}, y_{i,1}, \ldots, x_{i,k}, y_{i,k})$ for each $i \in [n]$, and if no pair $(x_{i,j},y_{i,j})$ is $(1,1)$, then $M$ corresponds to an odd cover of $G$ of cardinality $k$. If all four of $(0,0)$, $(0,1)$, $(1,0)$, and $(1,1)$ appear as pairs $(x_{i,j}, y_{i,j})$ in rows of $M$, then we can not associate to $M$ an odd cover of $G$. However, we can still interpret this case combinatorially by allowing complete tripartite graphs in our odd covers.

Let $G$ be a graph. We can generalize the notion of an odd cover as follows. Consider a collection of complete tripartite graphs, or tricliques, $\mathscr{T} = \{(X_1, Y_1, Z_1), \ldots, (X_k, Y_k, Z_k)\}$ on subsets of $V(G)$ in which two vertices are adjacent in an odd number of tricliques if and only if they are adjacent in $G$. That is, $E(G)$ is the symmetric difference of the edge sets of the tricliques in $\mathscr{T}$. (This notion is introduced in \cite{BPR} under the name tripartite subgraph complementation.)

We can encode such a collection of tricliques $\mathscr{T}$ in a familiar way. To each vertex $v_i$ of $G$, assign an incidence vector $(x_{i,1}, y_{i,1}, \ldots, x_{i,k}, y_{i,k})$ defined by
\[x_{ij} = \begin{cases}
1, & \mbox{if } v_i \in X_j \cup Z_j, \\
0, & \mbox{otherwise.}
\end{cases} \quad
y_{ij} = \begin{cases}
1, & \mbox{if } v_i \in Y_j \cup Z_j, \\
0, & \mbox{otherwise.}
\end{cases}\]
In other words, if $v_i \in X_j$, we have $x_{ij} = 1$ and $y_{ij} = 0$; if $v_i \in Y_j$, we have $x_{ij} = 0$ and $y_{ij} = 1$; and, if $v_i \in Z_j$, we have $x_{ij} = y_{ij} = 1$.
Let $M$ be the $n \times 2k$ matrix whose rows are the incidence vectors for the vertices of $G$. Once again, with $H_2 = \left(\begin{smallmatrix} 0 & 1 \\ 1 & 0 \end{smallmatrix}\right)$, we have
\[M \left(\oplus_1^k H_2\right) M^T = A(G).\]
Conversely, given a matrix $M$ as in Proposition~\ref{prop:friedland}, we can associate a collection of tricliques whose symmetric difference of edge sets is $E(G)$.

From Proposition~\ref{prop:friedland}, we can infer that for any graph $G$ there exists a collection of $r_2(G)/2$ tricliques whose symmetric difference of edge sets is $E(G)$. Since each triclique in such a collection can be replaced by two bicliques, we obtain an upper bound on $b_2(G)$ for any graph $G$: 
\begin{equation}\label{eq:2rankupper}
    b_2(G) \leq r_2(G).
\end{equation}
The idea of expressing graphs as symmetric differences of tricliques motivates us to define a new family of graphs, similar to the graphs $B_k$ in Definition~\ref{defn:Bk}.

\begin{definition}\label{defn:Tk}
    The graph $T_k$ has vertex set consisting of all strings of length $k$ with entries in $\{0,1,2, \blank\}$, where vertex $v$ and vertex $u$ are adjacent if and only if the number of places where they differ, and neither is $\blank$, is odd.
\end{definition}

The graphs $T_k - v$, where $v = (0, \ldots, 0)$, are in fact well-known. Letting $M$ be the $n \times (2^{2k}-1)$ matrix consisting of all distinct non-zero vectors over $\mathbb{F}_2$ of length $2k$, we have $A(T_k-v) = M(\oplus_1^k H_2)M^T$. In this light, we can recognize these as the non-orthogonality graphs of the unique non-degenerate symplectic form on $\mathbb{F}_2^{2k}$ studied in, for example, \cite{parsons1988exotic} and \cite{godsil2001chromatic}. The graphs $T_k$ have many interesting properties, one of which is similar to a property of $B_k$: every twin-free graph $G$ with $r_2(G) = 2k$ is an induced subgraph of $T_k$~\cite{godsil2001chromatic}. In particular, since the $2$-rank of $T_k-v$ is at most the rank of $M$, we have $r_2(T_k-v) = r_2(T_k) = 2k$. For $k \in \{1,2\}$, we have checked that $b_2(T_k) = 2k$ which agrees with the upper bound in equation~\eqref{eq:2rankupper}. The following proposition provides a lower bound for $b_2(T_k)$ which is bounded away from the general lower bound in Proposition~\ref{prop:ranklower} asymptotically.

\begin{prop}\label{prop:b2Tk}
For any positive integer $k$, $b_2(T_k) \geq \log_3(4) \cdot k$.
\end{prop}
\begin{proof}
Let $l = b_2(T_k)$. Since $T_k$ has no twin vertices, it is an induced subgraph of $B_l$ (see Definition~\ref{defn:Bk}). Thus, the order of $T_k$ is at most that of $B_l$, or $4^k \leq 3^l$. It follows that $l \geq \log_3(4) \cdot k$.
\end{proof}

Proposition~\ref{prop:b2Tk} implies the existence of a graph, for any positive integer $r$, such that $b_2(G) > r_2(G)/2 + r$. We can use the graphs $T_k$ to obtain other graphs with this property. Let $r$ and $k$ be positive integers such that $k \geq 2(\log_3(4) - 1)^{-1}r$, and let $H$ be a graph on a subset of the vertices of $T_k$ such that $b_2(H) \leq r/2$. By Proposition~\ref{prop:ranklower}, $r_2(H) \leq r$. If $G = T_k \triangle H$, then \[ \frac{r_2(G)}{2} \leq k+\frac{r}{2} \hspace{3mm} \text{and} \hspace{3mm} b_2(G) \geq \log_3(4)k - \frac{r}{2}. \]
It follows from our choice of $k$ that $b_2(G) \geq \frac{r_2(G)}{2} + r$.

It might be natural to think that one could construct a class of graphs for which $b_2(G)$ is bounded away from $r_2(G)/2$ asymptotically by taking disjoint unions of tricliques. However, while it is not hard to see that $r_2(G)$ is additive with respect to disjoint unions, $b_2(G)$ is not. Consider the disjoint union $2K_3$. We have $r_2(2K_3) = 4$, and, since $K_3$ is a triclique, it is not hard to find two tricliques whose symmetric difference is $2K_3$. Since the minimum cardinality of an odd cover of $K_3$ is 2, it is perhaps surprising that $b_2(2K_3) = 3$. Figure~\ref{fig:2K3} depicts a minimum odd cover of $G$. In fact, we can construct odd covers of $kK_3$ for any positive integer $k$ of cardinality $k+1 = r_2(kK_3)/2 + 1$, showing that the disjoint union of $k$ tricliques does not have $b_2(G)$ bounded away from $r_2(G)/2$.

\begin{figure}[h]
\begin{center}
\begin{tikzpicture}
[every node/.style={circle, draw=black!100, fill=black!100, inner sep=0pt, minimum size=4pt},
every edge/.style={draw, black!100, thick},
scale=.75]

\foreach \i in {1,...,6}
{
\node (\i) at (\i*360/6:1) {};
}
\draw[black!100,thick] (5) -- (6) -- (1) -- (5);
\draw[black!100,thick] (2) -- (3) -- (4) -- (2);

\node[draw=none,fill=none] at (2,0) {$=$};
\end{tikzpicture}
\quad
\begin{tikzpicture}
[every node/.style={circle, draw=black!100, fill=black!100, inner sep=0pt, minimum size=4pt},
every edge/.style={draw, black!100, thick},
scale=.75]

\foreach \i in {1,...,6}
{
\node (\i) at (\i*360/6:1) {};
}
\draw[black!100,thick] (6) -- (1) -- (2) -- (3) -- (6);

\node[draw=none,fill=none] at (2,0) {$\triangle$};
\end{tikzpicture}
\quad
\begin{tikzpicture}
[every node/.style={circle, draw=black!100, fill=black!100, inner sep=0pt, minimum size=4pt},
every edge/.style={draw, black!100, thick},
scale=.75]

\foreach \i in {1,...,6}
{
\node (\i) at (\i*360/6:1) {};
}
\draw[black!100,thick] (6) -- (3) -- (4) -- (5) -- (6);

\node[draw=none,fill=none] at (2,0) {$\triangle$};
\end{tikzpicture}
\quad
\begin{tikzpicture}
[every node/.style={circle, draw=black!100, fill=black!100, inner sep=0pt, minimum size=4pt},
every edge/.style={draw, black!100, thick},
scale=.75]

\foreach \i in {1,...,6}
{
\node (\i) at (\i*360/6:1) {};
}
\draw[black!100,thick] (1) -- (2) -- (4) -- (5) -- (1);
\end{tikzpicture}
\end{center}

\caption{A minimum odd cover of $2K_3$}
\label{fig:2K3}
\end{figure}
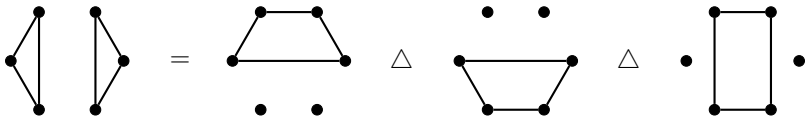

\section{Bipartite graphs}
\label{sec:bipartite}
We have seen in Proposition~\ref{prop:ranklower} that, in general, $b_2(G) \geq r_2 (G)/2$. In this section, we will show that $b_2(G) = r_2 (G)/2$ for all bipartite graphs $G$ by giving an explicit construction of an odd cover that achieves this lower bound. The method relies on Lemma~\ref{lem:nooddoddextension} below, which extends the observation from the previous paragraph that twin vertices do not affect $b_2(G)$.

We begin by proving the result for forests. We have seen that, in general, $b_2(G)$ and $\bp(G)$ do not agree. Figure~\ref{fig:C6}, for example, depicts a minimum odd cover of a bipartite graph for which $b_2(G) < \bp(G)$. However, in the case of forests, we do have $b_2(G) = \bp(G)$. To see this, we need the following lemma, which is a special case of Theorem~8 in~\cite{mohammadian2016trees}.

Let $m(G)$ denote the maximum size of a matching in $G$, and let $\tau(G)$ denote the minimum size of a vertex cover of $G$. It is a well-known result of K\H{o}nig~\cite{kHonig1931grafok} that, for any bipartite graph $G$, $m(G) = \tau(G)$.

\begin{lemma}[\cite{mohammadian2016trees}]\label{lem:treesmatching}
Let $T$ be a tree, and let $m(T)$ be the maximum size of a matching in $T$. Then $r_2(T) = 2m(T)$. 
\end{lemma}

It is not hard to see that the 2-rank of a graph is the sum of the 2-ranks of its components. Thus, $r_2(F) = 2m(F)$ for any forest $F$. We will show that $b_2(F) = m(F)$, which agrees with the lower bound in Proposition~\ref{prop:ranklower}.

\begin{prop}\label{prop:forests}
For any forest $F$, we have
\[b_2(F) = \bp(F) = m(F).\]
\end{prop}

\begin{proof}
Let $F$ be a forest. It is shown in~\cite{boyer1998rank} that $\bp(F) = m(F)$. Thus, $b_2(F) \leq m(F)$. By Proposition~\ref{prop:ranklower} and Lemma~\ref{lem:treesmatching}, we have $2b_2(F) \geq r_2(F) \geq 2m(F)$, from which the result follows.
\end{proof}

Proposition~\ref{prop:forests}, along with K\H{o}nig's theorem, implies that $b_2(F) = \tau(F)$. Given a vertex cover of a graph, we can easily construct a partition of its edges into stars. Thus, a partition of the edges of a forest $F$ into $\tau(F)$ stars is, in fact, a minimum odd cover of $F$.

We will proceed to prove that, for any bipartite graph $G$, $2b_2(G) = r_2(G)$. First, we require a lemma which holds for all graphs. Suppose that, for some vertex $v$ of a graph $G$, $r_2(G)=r_2(G-v)$. Then the row in $A(G)$ corresponding to the vertex $v$ can be written as the sum modulo 2 of a subset of the other rows of $A(G)$. Let $S$ be the set of vertices corresponding to these rows. In other words, the neighborhood of $v$ is the symmetric difference of the neighborhoods of the vertices in $S$: $vw \in E(G)$ if and only if $w$ has an odd number of neighbors in $S$.

\begin{lemma}\label{lem:nooddoddextension}
Suppose that $r_2(G-v) = r_2(G)$ for some vertex $v$ of a graph $G$.
Let $S$ be a set of vertices whose symmetric difference of open neighborhoods is $N(v)$.
If there is a minimum odd cover $\mathscr{B}$ of $G-v$ which does not contain a biclique whose partite sets each have an odd number of vertices from $S$, then $b_2(G) = b_2(G-v)$.
\end{lemma}

\begin{proof}
Suppose that we have a minimum odd cover of $G-v$ of cardinality $k=b_2(G-v)$, which we denote as 
\[ \mathscr{B} = \{ (X_1,Y_1),\dots,(X_k,Y_k)  \},  \] 
where each biclique $i$ is indicated by its two partite sets $X_i$ and $Y_i$. Suppose that there is no $i$ such that both $|S\cap X_i|$ and $|S\cap Y_i|$ are odd. Then, we can extend this odd cover to an odd cover of $G$ using $k$ bicliques as follows:
\[ \mathscr{B}' = \{ (X_1',Y_1'),\dots,(X_k',Y_k')  \},  \]
where
\[ X_i '= \begin{cases} X_i, & \mbox{if }|S\cap X_i|\equiv 0 \pmod 2,\\X_i\cup \{v\}, & \mbox{if }|S\cap X_i|\equiv 1 \pmod 2, \end{cases} \]
for $1 \leq i \leq k$, and the same for the sets $Y_i'$. Since $|S\cap X_i|$ and $|S\cap Y_i|$ are never odd at the same time, this is a valid construction, and it is easy to see that $vw$ is an edge in an odd number of bicliques in $\mathscr{B}'$ if and only if $w$ has an odd number of neighbors in $S$. That is, $\mathscr{B}'$ is an odd cover of $G$ of cardinality $k$.
\end{proof}

\begin{theorem}\label{thm:bipartite}
If $G$ is bipartite,
\begin{equation*}
 b_2(G) = \frac{r_2 (G)}{2}.   
\end{equation*}
Furthermore, there exists a minimum odd cover of $G$ that respects its bipartition. 
\end{theorem}

\begin{proof}
We proceed by induction on $n$, the order of $G$. The claim is easily verified for $n\leq 2$. Suppose that $G$ has $n+1$ vertices and the claim holds for $1,\dots,n$. Let $G=A \cup B$ be the bipartition, and let $v \in A$ be an arbitrary vertex of $G$. We distinguish two cases, based on the rank of $G-v$:
\begin{itemize}
\item[(i)] Suppose that $r_2(G)=r_2(G-v)+2$. We can extend a minimum odd cover $\mathscr{B}$ of $G-v$ to a minimum odd cover of $G$ by letting $\mathscr{B}'=\mathscr{B}\cup \{ (\{v\},N(v)) \}$.

\item[(ii)] Suppose that $r_2(G)=r_2(G-v)=k$. Then the neighborhood of $v$ is the sum (over $\mathbb{F}_2$) of the neighborhoods of a subset of vertices $S \subseteq A$. By the inductive hypothesis, there is a minimum odd cover $\mathscr{B}$ of $G-v$ that respects its bipartition. In particular, at least one partite set of each biclique in $\mathscr{B}$ contains no vertex in $S$. Thus, we may apply Lemma~\ref{lem:nooddoddextension} to extend $\mathscr{B}$ to a minimum odd cover of $G$.
\end{itemize}
This completes the proof.
\end{proof}

We obtain the following corollaries for certain classes of bipartite graphs for which the 2-rank is known.

\begin{cor}\label{cycle-even}
For $n$ even,
\[
b_2(C_n)=\frac{n-2}{2}.
\]
\end{cor}
\begin{cor}\label{path}
For any $n$,
\[
b_2(P_n)=\begin{cases}
\frac{n}{2}, & \mbox{if } n\equiv 0\pmod{2},\\
\frac{n-1}{2}, &\mbox{if } n\equiv 1\pmod{2}.
\end{cases}
\]
\end{cor}

\section{Odd cycles}
\label{sec:odd cycles}
We have seen that the lower bound in Proposition~\ref{prop:ranklower} is tight for even cycles. In addition, even cycles have the property $b_2(C_n)<\bp(C_n)$. However, neither of these hold for odd cycles. For odd $n$, $r_2(C_n)=n-1$, but we will show that $b_2(C_n)=\bp(C_n)=\frac{n+1}{2}$. 

\begin{theorem}\label{cycle-odd} Let $C_n$ be an odd cycle on $n$ vertices. Then $b_2(C_n)=\frac{n+1}{2}$.
\end{theorem}

\noindent \textit{Proof.} Using an edge-disjoint collection of $\frac{n-1}{2}$ copies of $K_{1,2}$ and one $K_{1,1}$, we can find an odd cover of $C_n$, as is shown in Figure~\ref{fig:C5} for $C_5$. This establishes the upper bound.

Consider an odd cover, $\mathscr{B}$, of $C_n$. Since $C_n$ has an odd number of edges, there exists a biclique $B:=(X,Y)$ in $\mathscr{B}$ where $|X|=a$ and $|Y|=b$ are both odd. We will demonstrate that $r_2(C_n\triangle B)\geq{n-1}$, which implies that $b_2(C_n)-1\geq{\frac{n-1}{2}}$. 

Row $i$ of the adjacency matrix for $C_n$ is given by $f_i:=e_{i-1}+e_{i+1}$ where the indices are viewed modulo $n$.
Let $v_X=\sum_{j\in X}e_j$ and $v_Y=\sum_{j\in Y}e_j$. The $i^{th}$ row of the adjacency matrix for $B$ is $v_Y$ for $i\in X$, $v_X$ for $i\in Y$, and $0$ for the remaining rows.

Let $S$ be the vector space over $\mathbb{F}_2$ spanned by the rows of $A(C_n\triangle B)$ with indices in $X$, $T$ be the vector space over $\mathbb{F}_2$ spanned by the rows of $A(C_n \triangle B)$ with indices in $Y$, and $U$ be the vector space over $\mathbb{F}_2$ spanned by the remaining $n-a-b>0$ rows of $A(C_n\triangle B)$. Then
\begin{align*}
    r_2(C_n\triangle B_k)&=\dim{(S+T+U)}\\&=\dim{((S+T)+U)}\\
    &=\dim{(S+T)}+\dim{U}-\dim{((S+T)\cap U)}\\
    &=\dim{S}+\dim{T}-\dim(S\cap T)+\dim{U}-\dim((S+T)\cap U).
\end{align*}
To establish that this quantity is at least $n-1$, we make use of the following lemmas:

\begin{lemma}\label{cyc1}
$\dim{S}=a$, $\dim{T}=b$, and  $\dim{U}=n-a-b$.
\end{lemma}

\begin{lemma}\label{cyc2}
There are at most two vectors in $S\cap T$ and at most three vectors in $(S+T)\cap U$.
\end{lemma}

Before proving these lemmas, we will show why proving them is sufficient to prove Theorem \ref{cycle-odd}. Lemma \ref{cyc2} gives that  $\dim(S\cap T)\leq{1}$. Since the number of vectors in each vector space over $\mathbb{F}_2$ must be a power of $2$, there are at most two vectors in $(S+T)\cap U$, so $\dim((S+T)\cap U)\leq{1}$.
Therefore $\dim(S+T+U)\geq{a+b-1+(n-a-b)-1}=n-2$. However, this is the rank of an adjacency matrix over $\mathbb{F}_2$ so it must be even (see \cite{brouwer1992}). Therefore, $r_2(C_n\triangle B)$ is at least $n-1$, as desired. It follows that $b_2(C_n)\geq{\frac{n+1}{2}}$, matching the upper bound.

\begin{proof}[Proof of Lemma \ref{cyc1}]
Note that any proper subset of the rows of $A(C_n)$ is linearly independent, so the $n-a-b$ rows whose span is used to define $U$ are linearly independent. Thus, $\dim(U)=n-a-b$. We also claim that $\{f_j+v_Y\mid j\in X\}$ and $\{f_j+v_X\mid j\in Y\}$ are linearly independent sets. Suppose there exist constants $c_j\in\mathbb{F}_2$ for $j\in X$ such that \[\sum_{j\in X}c_j(f_j+v_Y)=0.\]
Let $C = \sum_{j \in X} c_j$. If $C=0$, then we have that 
$\sum_{j\in X}c_jf_j=0$. 
However, the $f_j$'s for $j\in X$ are a proper subset of the rows of $A(C_n)$, so this can only occur when each $c_j$ is $0$. 

If instead we have $C=1$, then the sum $\sum_{j\in X}c_j(f_j+v_Y)=0$ simplifies to $\sum_{j\in X}c_jf_j=v_Y$. This is impossible since the sum of the components of $\sum_{j\in X}c_jf_j$ is even for all choices of $c_j$'s while the sum of the components of $v_Y$ is odd since $b$ is odd. This means that $\{f_j+v_Y \mid j\in X\}$ is linearly independent. By the same reasoning, $\{f_j+v_X \mid j\in Y\}$ is also linearly independent. This establishes that $\dim(S)=a$ and $\dim(T)=b$.
\end{proof}

\begin{proof}[Proof of Lemma \ref{cyc2}]
First, we consider a vector $w \in S\cap T$. As such
\[w = \sum_{j\in X}c_j(f_j+v_Y)=\sum_{j\in Y}d_j(f_j+v_X)\] for some choices of $c_j,d_j\in\mathbb{F}_2.$ Let $C = \sum_{j \in X} c_j$ and $D = \sum_{j \in Y} d_j$. As the sum of the components of $\sum_{j\in X}c_jf_j+\sum_{j\in Y}d_jf_j$ is always even, we have $C=D$.

If $C=D=0$, then
\[ \sum_{j\in X}c_jf_j =\sum_{j\in Y}d_jf_j \hspace{3mm} \text{which implies} \hspace{3mm} \sum_{j\in X}c_jf_j+\sum_{j\in Y}d_jf_j =0. \]

Since $\{f_j\mid j\in X\}\cup\{f_j\mid j\in Y\}$ is a proper subset of the rows of $A(C_n)$, it is linearly independent and hence $c_j=0$ for all $j\in X$ and $d_j=0$ for all $j\in Y$. Hence $w$ is necessarily the zero vector in this case. 

Lastly, if $C=D=1$, we get 
\[\sum_{j\in X}c_jf_j+\sum_{j\in Y}d_jf_j=v_X+v_Y. \]

Since $\{f_j\mid j\in X\}\cup\{f_j\mid j\in Y\}$ is a proper subset of rows of $A(C_n)$, it is linearly independent. As such, the choices of $c_j$ and $d_j$ to write $v_X+v_Y$ are unique and thus there is at most one such $w$. Hence, there are at most two vectors in $S\cap T$. 

Now, let $ w \in (S+T)\cap U$. As a result, 
\begin{equation}\label{eq:oddcyclew}
w = \sum_{j\in X}c_j(f_j+v_Y)+\sum_{j\in Y}d_j(f_j+v_X)=\sum_{j\in [n]\setminus{(X\cup Y)}}r_jf_j
\end{equation}
for some choices of $c_j,d_j,r_j\in\mathbb{F}_2.$ Again, let $C = \sum_{j \in X} c_j$ and $D = \sum_{j \in Y} d_j$.

Note that we must again have $C=D$; otherwise, the sum of the components on the left side of the second equality in~\eqref{eq:oddcyclew} would be odd while the sum of the components on the right side would be even.

If $C=D=0$, we have
\[ \sum_{j\in X}c_jf_j+\sum_{j\in Y}d_jf_j+\sum_{j\in[n]\setminus (X\cup Y)}r_jf_j =0. \] 
Since $A(C_n)$ has rank $n-1$, we know that the dimension of $\spn\{f_j\mid j\in [n]\}$ is $n-1$. By Rank-Nullity, there are exactly two choices for the $c_j$'s, $d_j$'s, and $r_j$'s that satisfy this equation.
Either all coefficients are $0$ or all coefficients are $1$. However, picking all the coefficients to be $1$ violates the conditions $\sum_{j\in X}c_j=0$ and $\sum_{j\in Y}d_j=0$ since $|X|=a$ and $|Y|=b$ are odd. Thus, all coefficients are $0$ and this case does not contribute any nonzero vectors to $(S+T)\cap U$. 

The only remaining case to consider is when $C=D=1$. This leads to
\[ \sum_{j\in X}c_jf_j+\sum_{j\in Y}d_jf_j+\sum_{j\in[n]\setminus (X\cup Y)}r_jf_j=v_X+v_Y. \]
Again, since $r_2(C_n)=n-1$, we know that if $v_X+v_Y$ is in the span of $\{f_j\mid j\in[n]\}$, then there are exactly two ways to pick the coefficients $c_j$, $d_j$, $r_j$. These correspond to at most two vectors in $(S+T)\cap U$. Along with the zero vector from the $C=D=0$ case, we see that there are at most three vectors in $(S+T)\cap U$.
\end{proof}

\section{Complete graphs}\label{complete}
In this section we will show that for all complete graphs $K_n$, $b_2(K_n)$ is either $\lceil n/2\rceil$ or $\lceil n/2\rceil+1$. The lower bound comes from evaluating the $2$-rank and the upper bound comes from explicit constructions. In particular, when $n=8k$ or $8k\pm 1$ for some positive integer $k$, $b_2(K_n)=\lceil n/2\rceil$. 

Recall that two vertices $u$ and $v$ of a graph $G$ are said to be {\em adjacent twins} if $N[u]=N[v]$. We will make use of the following lemma to bound $b_2(K_n)$.

\begin{lemma}\label{lemma:RankByAdjacentTwins}
Let $G$ be a graph on $n$ vertices. If $G$ contains a matching $M$ such that each edge $uv\in M$ is a pair of adjacent twins, then $r_2(G)=2|M|+r_2(G - V(M))$.
\end{lemma}

\begin{proof}
It suffices to show that, by elementary operations, the adjacency matrix of $G$ can be turned into a block diagonal matrix whose blocks are an identity matrix of size $2|M|$ and the adjacency matrix of $G - V(M)$. Without loss of generality, we assume the first two rows $r_1,r_2$ correspond to two vertices $v_1,v_2$ that form an edge in $M$. By definition, $r_1+r_2=(1,1,0,\dots,0)$. The first two entries in each row other than $r_1$ and $r_2$ is either $(1,1)$ or $(0,0)$. So we can turn all entries in the first two columns, except for the two diagonal entries, into 0 by elementary row operations. And then we can turn all entries in the first two rows, except for the two diagonal entries, into 0 by elementary column operations. Similarly, assuming that the first $2|M|$ rows and columns correspond to vertices in $V(M)$, we can turn all entries in the first $2|M|$ rows or the first $2|M|$ columns, except for the $2|M|$ diagonal entries, into 0 by elementary operations, while the entries in the last $(n-2|M|)\times (n-2|M|)$ diagonal block remain the same. This completes the proof. 
\end{proof}

\begin{lemma}\label{lem:oddKnlower}
For integers $k\ge 1$, $b_2(K_{2k+1})\ge k+1$.
\end{lemma}

\begin{proof}
Let $G$ be an arbitrary biclique, which is a subgraph of $K_{2k+1}$. It suffices to show that $r_2(K_{2k+1}\triangle G)=2k$. Let $A$ and $B$ be the two parts of $G$. If the size of one of $A$ and $B$ is even, then it is not hard to find $k$ vertex-disjoint pairs of adjacent twins in $K_{2k+1}\triangle G$, and hence, by Lemma~\ref{lemma:RankByAdjacentTwins}, $r_2(K_{2k+1}\triangle G)=2k$ as desired. If both $A$ and $B$ have odd size, then we can find a matching $M$ consisting of $k-1$ vertex-disjoint pairs of adjacent twins, such that $(K_{2k+1}\triangle G)- V(M)$ is a path of length $2$, which has $2$-rank $2$. Therefore, by Lemma~\ref{lemma:RankByAdjacentTwins}, we have $r_2(K_{2k+1}\triangle G)=2k-2+2=2k$.
\end{proof}

In fact, we can also prove this lemma by analyzing the adjacency matrix directly, similar to what we did for odd cycles.

When $n$ is an even integer, we can find an odd cover using at most $n/2+1$ bipartite graphs for a family of $n$-vertex graphs which contains the complete graph $K_n$. 

\begin{lemma}\label{lem:AdjacentTwins}
Let $G$ be a graph on $n$ vertices. If $G$ contains a perfect matching $M$ such that each edge $uv \in M$ is a pair of adjacent twins, then $r_2(G)=n$ and $b_2(G)\le n/2+1$. 
\end{lemma}

\begin{proof}
By Lemma~\ref{lemma:RankByAdjacentTwins}, we have $r_2(G)=n$. We can show $b_2(G)\le n/2+1$ by construction. Let $M=\{a_ib_i\}_{1\le i\le n/2}$. Let $G^{(j)}$ be the induced subgraph of $G$ on vertices $\{a_1,b_1,a_2,b_2, \dots,a_j,b_j\}$. We prove the following statement by induction: For any positive integer $j$, if $j$ is odd, then there exists a set of vectors $\{a^{(1)}, b^{(1)},a^{(2)},b^{(2)},\dots,a^{(j)},b^{(j)}\}\subset \{0,1,\blank\}^{j+1}$ with the following properties.
\begin{enumerate}
    \item The induced subgraph of $B_{j+1}$ (See Definition~\ref{defn:Bk}) on this set of vectors is isomorphic to $G^{(j)}$.
    \item $a^{(t)}_i+b^{(t)}_i=1$ for all $1\le t\le j$, $i\neq t$. 
    \item $a^{(t)}_i=b^{(t)}_i=\blank$ if and only if $i=t$.
\end{enumerate}
If $j$ is even, then there exists a set of vectors $\{a^{(1)}, b^{(1)},a^{(2)},b^{(2)},\dots,a^{(j)},b^{(j)}\}\subset \{0,1,\blank\}^{j+2}$ with the following properties.
\begin{enumerate}
    \item The induced subgraph of $B_{j+2}$ (See Definition~\ref{defn:Bk}) on this set of vectors is isomorphic to $G^{(j)}$.
    \item $a^{(t)}_i+b^{(t)}_i=1$ for all $1\le t\le j$, $i\neq t$.
    \item $a^{(t)}_i=b^{(t)}_i=\blank$ if and only if $i=t$.
    \item $a^{(t)}_{j+2}=0$ for all $1\le t\le j$.
\end{enumerate}
The $0,1$ entries of the vectors are considered as numbers in modulo $2$. Let $I(i,j)$ be the indicator function of whether $a_ia_j$ forms an edge, that is,
$$
I(i,j)=
\left\{
\begin{array}{l}
     1,\quad \text{if $a_ia_j\in E(G)$,}  \\
     0,\quad \text{otherwise.} 
\end{array}
\right.
$$
When $j=1$, let $a^{(1)}=(\blank,0)$, and $b^{(1)}=(\blank,1)$.  When $j=2$, let $a^{(1)}=(\blank, 0, 0, 0)$, $b^{(1)}=(\blank, 1, 1, 1)$,  $a^{(2)}=(0,\blank, I(1,2),0)$, and $b^{(2)}=(1,\blank, 1-I(1,2), 1)$. It is easy to check that they satisfy the desired properties. Suppose for some integer $k\ge 1$, we have constructed a set of vectors $\{a^{(1)}, b^{(1)},a^{(2)},b^{(2)},\dots,a^{(2k)},b^{(2k)}\}\in \{\blank,0,1\}^{2k+2}$ with the desired properties.

To prove the statement for $j=2k+1$, let
\[
S_1(a^{(t)})=\sum_{\substack{1\le i\le 2k\\i\neq t}}a_i^{(t)},
\]
for $1\le t\le 2k$. Create a new vector $a^{(2k+1)}\in\{0,1,\blank\}^{2k+2}$. Set
\[
\begin{aligned}
&a^{(2k+1)}_{2k+1}=\blank,\\
&a^{(2k+1)}_i=S_1(a^{(i)})+I(i,2k+1),~~~ \text{for}\ 1\le i\le 2k,\\
&a^{(2k+1)}_{2k+2}=\sum_{i=1}^{2k}a^{(2k+1)}_i.
\end{aligned}
\]
Let $b^{(2k+1)}$ be the complement of $a^{(2k+1)}$; that is, replace 1 with 0 and 0 with 1. We can now check that 
\[
\{a^{(1)}, b^{(1)},a^{(2)},b^{(2)},\dots,a^{(2k+1)},b^{(2k+1)}\}\in \{\blank,0,1\}^{2k+2}
\]
is a set of vectors with desired properties. It suffices to check that for each $1\le j\le 2k$, the number of coordinates such that one is $0$ and the other is $1$ is $I(j,2k+1)$ between $a^{(j)}$ and $a^{(2k+1)}$, which is equivalent to
\[
\sum_{\substack{1\le i\le 2k+2\\i\neq j,i\neq 2k+1}}(a_i^{(j)}+a_i^{(2k+1)})=I(j,2k+1).
\]

To prove the statement for $j=2k+2$, let
\[
S_2(a^{(t)})=\sum_{\substack{1\le i\le 2k+1\\i\neq t}}a_i^{(t)},
\]
for $1\le t\le 2k+1$. Give each vector $a^{(i)}$ two extra $0$'s at the end, and give each $b^{(i)}$ two extra $1$'s at the end, for $1\le i\le 2k+1$. Create a new vector $a^{(2k+2)}\in\{0,1,\blank\}^{2k+4}$. Set
\[
\begin{aligned}
&a^{(2k+2)}_{2k+2}=\blank,\\
&a^{(2k+2)}_i=S_2(a^{(i)})+I(i,2k+2),~~~ \text{for}\ 1\le i\le 2k+1,\\
&a^{(2k+2)}_{2k+3}=\sum_{i=1}^{2k+1}a^{(2k+2)}_i,\\
&a^{(2k+2)}_{2k+4}=0.
\end{aligned}
\]
Let $b^{(2k+2)}$ be the complement of $a^{(2k+2)}$. We can now check that 
\[
\{a^{(1)}, b^{(1)},a^{(2)},b^{(2)},\dots,a^{(2k+2)},b^{(2k+2)}\}\in \{\blank,0,1\}^{2k+4}
\]
is a set of vectors with desired properties. This completes the proof of the statement.

When $n/2$ is odd, the statement above for odd values provides us with an odd cover of $G$ using $n/2+1$ bicliques. This implies $b_2(G)\le n/2+1$. On the other hand, when $n/2$ is even, the statement above for even values only provides us with an odd cover of $G$ using $n/2+2$ bicliques. However, we notice that the union of the two bicliques corresponding to the last two columns is also a biclique. Therefore, this also implies $b_2(G)\le n/2+1$. 
\end{proof}

Now we are ready to prove our main result in this section.
\begin{theorem}\label{clique}
For any positive integer $n$,
\[
\left\lceil\frac{n}{2}\right\rceil \leq b_2(K_n) \leq \left\lceil\frac{n}{2}\right\rceil + 1.
\]
In particular, for any positive integer $k$
\[
\begin{aligned}
&b_2(K_{8k-1})=b_2(K_{8k})=4k,\\
&b_2(K_{8k+1})=4k+1.
\end{aligned}
\]
\end{theorem}
\begin{proof}
Lemma~\ref{lemma:RankByAdjacentTwins} gives $r_2(K_{2n})=2n$ and hence $b_2(K_{2n})\ge n$. Together with Lemma~\ref{lem:oddKnlower}, we have the general lower bounds 
\[
b_2(K_n)\ge \left\lceil\frac{n}{2}\right\rceil.
\]
For upper bounds, Lemma~\ref{lem:AdjacentTwins} gives $b_2(K_{2n})\le n+1$. Since $b_2(K_{2n+1})\le b_2(K_{2n})+1\le n+2$, we have the general upper bounds 
\[b_2(K_n)\le \left\lceil\frac{n}{2}\right\rceil+1.\] 

For $n=8k$, it suffices to show that $b_2(K_{8k})\le 4k$. Note that this also implies $b_2(K_{8k-1})\le 4k$ and $b_2(K_{8k+1})\le 4k+1$.  We will construct explicit odd covers of $K_{8k}$ represented by vectors as vertices in $B_{4k}$ (See Definition~\ref{defn:Bk}, also see Figure~\ref{fig:K_32} for the construction of an odd cover of $K_{32}$). Let $\{a^{(i)}\}_{1\le i\le 4k}\subset\{0,1,\blank\}^{4k}$ be a set of vectors defined by:
\begin{enumerate}
    \item $a^{(i)}_i=\blank$;
    \item $a^{(i)}_{j}=0$ if $\left\{\begin{array}{l}
    j\ge i+2 \\
    i\equiv 0,1\mod 4 \text{ and } j=i+1 \\
    i\equiv 0,3\mod 4 \text{ and } j=i-1
    \end{array}\right.$;
    \item $a^{(i)}_{j}=1$ otherwise.
\end{enumerate}

\begin{figure}[h]
    \centering
    \includegraphics[scale=0.2]{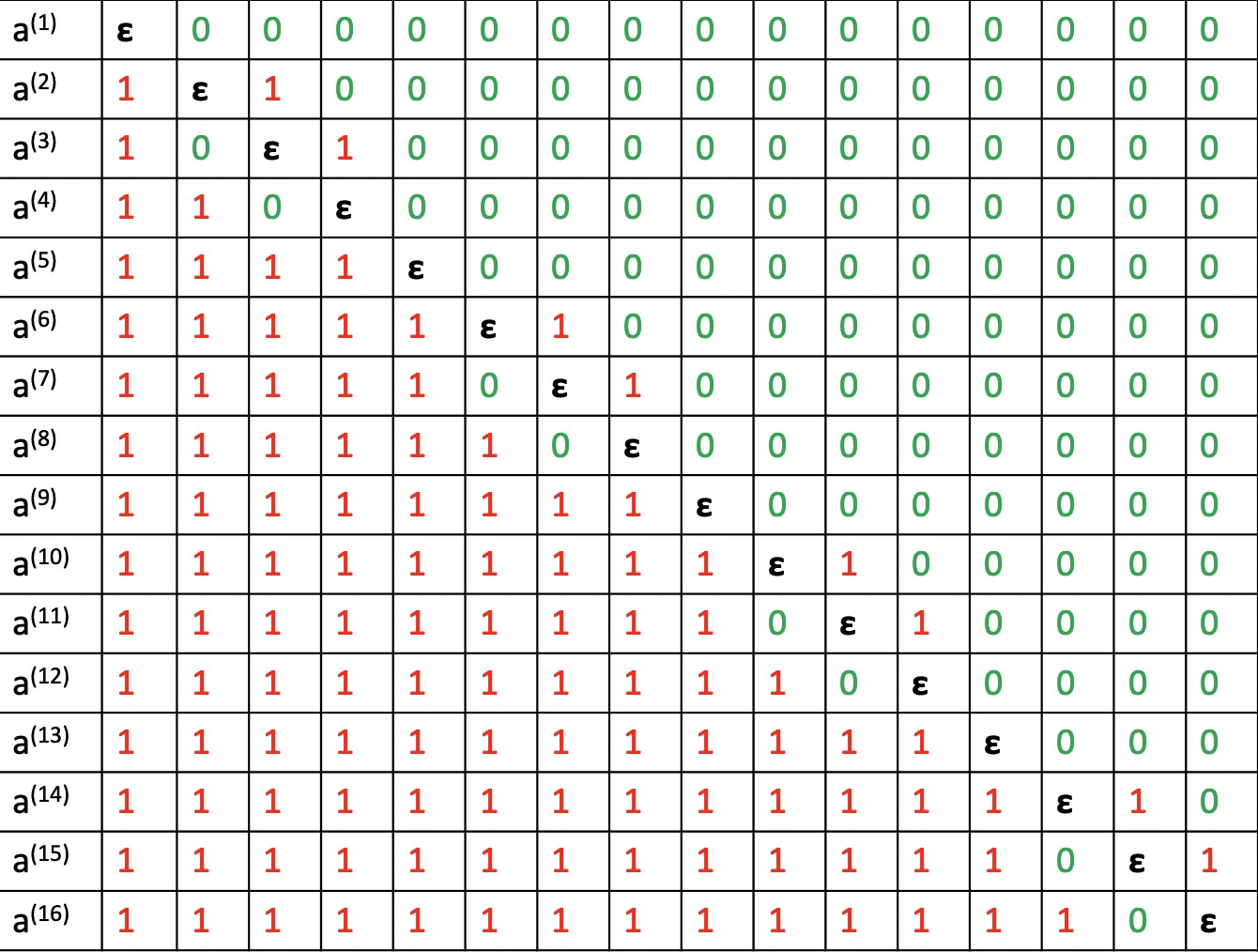}
    \caption{Vectors $a^{(i)}$ for $K_{32}$}
    \label{fig:K_32}
\end{figure}

Also let $\{b^{(i)}\}_{1\le i\le 4k}$ be a set of vectors such that $b^{(i)}$ is the complement of $a^{(i)}$ for all $i$ (replace each $1$ by $0$ and each $0$ by $1$). Let $S$ denote the set of these $8k$ vectors. We can check that $S$ gives an odd cover of $K_{8k}$. Indeed, one can check that the following properties hold:
\begin{enumerate}
    \item $\{a^{(i)},b^{(i)}\}$ form an edge for all $1\le i\le 4k$.
    \item $\{a^{(i)},v\}$ form an edge if and only if $\{b^{(i)},v\}$ form an edge for all $v\in S\setminus\{a_i,b_i\}$, $1\le i\le 4k$.
    \item $\{a^{(i)},v\}$ form an edge if and only if $\{a^{(i+4)},v\}$ for all $v\in S\setminus\{a^{(i)},a^{(i+4)}\}$ $1\le i\le 4k-4$.
\end{enumerate}
Since $a^{(1)},a^{(2)},a^{(3)}$ and $a^{(4)}$ form a $K_4$, the proof is complete.
\end{proof}

\section{Conclusion}
\label{conclusion}
\begin{itemize}[leftmargin=*]
\item In this paper, we determined that $b_2(K_n)=\lceil n/2 \rceil$ when $n \in \{0,1,7\} \pmod{8}$. As a result, one might hope that the behavior of $b_2(K_n)$ depends on the value of $n \pmod{8}$. Our computer-aided calculations when $n \leq 14$ together with Theorem~\ref{clique} determine the values of $b_2(K_n)$ for $1\leq n \leq 17$. Although these computer-aided calculations yield that $b_2(K_{10}) = 6 = \lceil 10/2 \rceil +1$ whereas $b_2(K_{2}) = 1 = \lceil 2/2 \rceil$, we believe this behavior is unique to the case when $n=2$:

\begin{conj}\label{conj:evenvalues}
Let $n \geq 2$ be an integer. Then 
\[ b_2(K_{2n}) = 
\begin{cases}
n \hspace{3mm} \text{when} \hspace{3mm}  4|n \\
n +1 \hspace{3mm}  \text{otherwise} \hspace{3mm} 
\end{cases}
\]
\end{conj}

It is worth noting the connection between Conjecture~\ref{conj:evenvalues}, the existence of Hadamard matrices, and the connection between Hadamard matrices and $b_2(K_n)$ established in~\cite{radhakrishnan2000}. An $n \times n$ Hadamard matrix is conjectured to exist (when $n \geq 3$) if and only if $4|n$. Accepting this conjecture and using the implication from~\cite{radhakrishnan2000}, we get that $b_2(K_{2n}) = n$ when $4|n$. Conjecture~\ref{conj:evenvalues} implies this condition is both necessary and sufficient; namely, if there does not an $n \times n$ Hadamard matrix, then $b_2(K_{2n})=n+1$. However, our proof that $b_2(K_{8n})=4n$ ({\em i.e.}, the first line of Conjecture~\ref{conj:evenvalues}) is independent of the existence of Hadamard matrices, and as such it may be possible to prove Conjecture~\ref{conj:evenvalues} with new methods also unrelated to Hadamard matrices. In the case for odd cliques, Theorem~\ref{clique} establishes that $b_2(K_{2n-1}) = n$ when $2n-1 = \pm 1 \pmod{8}$, and we conjecture this also holds when $2n-1 \in \{3,5\} \pmod{8}$:

\begin{conj}\label{conj:oddvalues}
Let $n \geq 2$ be an integer. Then 
\[ b_2(K_{2n-1}) = n  \]
\end{conj}

\item We also explored the minimum cardinality of an odd cover of $G$ for a variety of different types of graphs. For all $n$ vertex graphs that we consider, it follows that $b_2(G) \leq n/2 +1$. Although this may be too bold of a phenomenon to hold in general, it would be interesting to show that the minimum cardinality of an odd cover of an $n$-vertex graph is asymptotically bounded away from $n$: 

\begin{prob}\label{prob:awayfrom1}
Do there exist $\epsilon>0$ and $n_0 \in \mathbb{N}$ with the following property: for all $n>n_0$ and every $n$-vertex graph $G$, $b_2(G) \leq (1-\epsilon)n$?
\end{prob}

Perhaps a step towards an affirmative answer to Problem~\ref{prob:awayfrom1} is to consider the minimum cardinality of an odd cover of the random graph $G_{n,1/2}$. Letting $\alpha(G)$ denote the largest size of an independent set in $G$, and by taking a collection of appropriately chosen stars, it follows that $b_2(G) \leq n-\alpha(G)$ and similarly $\bp(G) \leq n-\alpha(G)$. The problem of biclique partitions of the random graph $G_{n,1/2}$ is a well-studied problem where the best current bounds are due to Alon, Bohman, and Huang~\cite{abh} who recently showed that there exists an absolute constant $c > 0$ so that $ \bp(G_{n, 1/2}) \leq n - (1+c) \alpha(G_{n,1/2})$ with high probability.

\item 
Our main source of lower bounds on $b_2(G)$ comes from utilizing the $2$-rank and subadditivity of matrix rank. We also can explore a hypergraph extension of odd covers. Let $K_n^{(r)}$ denote the complete $r$-uniform hypergraph on vertex set $[n]:=\{1,2, \ldots, n\}$. Given pairwise disjoint $X_1, \ldots, X_r \subset [n]$, let $\prod_{i=1}^r X_i$ denote the complete $r$-partite $r$-uniform hypergraph whose hyperedges consist of all $r$-sets which contain exactly one vertex from each $X_i$. Let $b_r(K_n^{(r)})$ denote the \textit{minimum} number of complete $r$-uniform $r$-partite hypergraphs needed so that each edge of $K_n^{(r)}$ appears in an \textit{odd number} of these complete $r$-uniform $r$-partite hypergraphs. A natural upper bound on $b_r(K_n^{(r)})$ comes from the minimum number of complete $r$-uniform $r$-partite hypergraphs needed to partition $K_n^{(r)}$, which we denote as $\bp_r(K_n^{(r)})$.  

\medskip 
Taking an arbitrary odd cover of $K_n^{(2r)}$ and using the methods of Cioab\u{a}, K\"{u}ngden, and Verstra\"{e}te~\cite{CKV}, we recover an odd cover of the \textit{Kneser graph} $K_{n:r}$, {\em i.e.}, $V(K_{n:r}) = \binom{[n]}{r}$ and $(A,B) \in E(K_{n:r})$ if and only if $A \cap B = \emptyset$. Using an argument of Wilson~\cite{W} on the $2$-rank of Kneser graphs as in~\cite{OV} together with the best general bounds on $\bp_r(K_n^{(r)})$ due to Leader, Mili\'{c}evi\'{c}, and Tan~\cite{LMT}, we recover 
\begin{equation}\label{eq:hypergraphcliquesymdiff}
\bigg( \frac{1}{ \binom{2r}{r}} + o(1) \bigg) \binom{n}{r} \leq b_{2r}(K_n^{(2r)}) \leq \bigg(\frac{14}{15} + o(1) \bigg) \binom{n}{r}.    
\end{equation}

In this paper, we showed that $b_2(K_n) < \bp_2(K_n)=n-1$  for $n\geq 5$, and it would be interesting to show the two quantities differ for larger uniformities. 

\begin{prob}\label{prob:hypergraphanalog}
For $r \geq 3$, is it the case that $b_r(K_n^{(r)}) < \bp_r(K_n^{(r)})$ for $n$ sufficiently large? 
\end{prob}

\end{itemize}

\section{Acknowledgements}

The authors would like to thank Daniela Ferrero and Kate Lorenzen for helpful conversations. Part of this research was conducted at the Graduate Research Workshop in Combinatorics 2021. Alexander Clifton was partially supported by NSF award DMS-1945200. Jiaxi Nie and Jason O'Neill were partially supported by NSF award DMS-1800332. Mei Yin was partially supported by the University of Denver's Faculty Research Fund 84688-145601.

\bibliographystyle{plain}

\begin{thebibliography}{10}

\bibitem{abh}
Noga Alon, Tom Bohman, and Hao Huang.
\newblock More on the bipartite decomposition of random graphs.
\newblock {\em Journal of Graph Theory}, 84(1):45--52, 2017.

\bibitem{babai2020}
L\'{a}szl\'{o} Babai and P\'{e}ter Frankl.
\newblock {\em Linear Algebra Methods in Combinatorics (with applications to
  Geometry and Computer Science)}.
\newblock Department of Computer Science, The University of Chicago, 2020.

\bibitem{boyer1998rank}
Elizabeth~D. Boyer.
\newblock Rank and biclique partitions of the complement of paths.
\newblock {\em Journal of Graph Theory}, 27(3):111--122, 1998.

\bibitem{brouwer1992}
A.~E. Brouwer and C.~A.~van Eijl.
\newblock On the $p$-rank of the adjacency matrix of strongly regular graphs.
\newblock {\em Journal of Algebraic Combinatorics}, 1(4):329--346, 1992.

\bibitem{BPR}
Calum Buchanan, Christopher Purcell, and Puck Rombach.
\newblock Subgraph complementation and minimum rank.
\newblock {\em preprint, arXiv:2101.06180}, 2021.

\bibitem{CKV}
Sebastian Cioab\u{a}, Andr\'{e} K\"{u}ndgen, and Jacques Verstra\"{e}te.
\newblock On decompositions of complete hypergraphs.
\newblock {\em Journal of Combinatorial Theory, Series A}, pages 1232--1234,
  2009.

\bibitem{76043}
Niel de~Beaudrap~(https://mathoverflow.net/users/3723/niel-de-beau drap).
\newblock Decomposition of graphs as symmetric differences of copies of
  ${K}_{a,b}$.
\newblock MathOverflow.
\newblock URL:https://mathoverflow.net/q/76043 (version: 2011-10-27).

\bibitem{Launey}
W.~de~Launey and D.~M. Gordon.
\newblock On the density of the set of known {H}adamard orders.
\newblock {\em Cryptography and Communications}, 2:233--246, 2010.

\bibitem{friedland1991quadratic}
Shmuel Friedland.
\newblock Quadratic forms and the graph isomorphism problem.
\newblock {\em Linear Algebra and its Applications}, 150:423--442, 1991.

\bibitem{godsil2001chromatic}
Chris~D Godsil and Gordon~F Royle.
\newblock Chromatic number and the 2-rank of a graph.
\newblock {\em Journal of Combinatorial Theory, Series B}, 81(1):142--149,
  2001.

\bibitem{GP}
Ron Graham and Henry Pollak.
\newblock On embedding graphs in squashed cubes.
\newblock {\em Graph Theory and Appl., Springer Lecture Notes in Math.},
  303:99--110, 1972.

\bibitem{KAMINSKI20092747}
Marcin Kamiński, Vadim~V. Lozin, and Martin Milanič.
\newblock Recent developments on graphs of bounded clique-width.
\newblock {\em Discrete Applied Mathematics}, 157(12):2747--2761, 2009.
\newblock Second Workshop on Graph Classes, Optimization, and Width Parameters.

\bibitem{kHonig1931grafok}
D{\'e}nes K{\H{o}}nig.
\newblock Gr{\'a}fok {\'e}s m{\'a}trixok.
\newblock {\em Matematikai {\'e}s Fizikai Lapok}, 38:116--119, 1931.

\bibitem{LMT}
Imre Leader, Luka Mili\'{c}evi\'{c}, and Ta~Shang Tan.
\newblock Decomposing the complete $r$-graph.
\newblock {\em Journal of Combinatorial Theory, Series A}, 154:21--31, 2018.

\bibitem{mohammadian2016trees}
Ali Mohammadian.
\newblock Trees and acyclic matrices over arbitrary fields.
\newblock {\em Linear and Multilinear Algebra}, 64(3):466--476, 2016.

\bibitem{OV}
Jason O'Neill and Jacques Verstra\"{e}te.
\newblock A note on $k$-wise oddtown problems.
\newblock {\em arxiv:2011.09402}, 2020.

\bibitem{parsons1988exotic}
Torrence~Douglas Parsons and Toma{\v{z}} Pisanski.
\newblock Exotic n-universal graphs.
\newblock {\em Journal of Graph Theory}, 12(2):155--158, 1988.

\bibitem{radhakrishnan2000}
Jaikumar Radhakrishnan, Pranab Sen, and Sundar Vishwanathan.
\newblock Depth-3 arithmetic circuits for ${S}_{n}^{2}(x)$ and extensions of
  the {G}raham-{P}ollack [\textit{sic}] theorem.
\newblock In Sanjiv Kapoor and Sanjiva Prasad, editors, {\em FST TCS 2000:
  Foundations of Software Technology and Theoretical Computer Science. FSTTCS
  2000. Lecture Notes in Computer Science}, volume 1974. Springer, Berlin,
  Heidelberg, 2000.

\bibitem{W}
Richard Wilson.
\newblock A diagonal form for the incidence matrices of $t$-subsets vs.
  $k$-subsets.
\newblock {\em European Journal of Combinatorics}, 11:609--615, 1990.

\end{thebibliography}

\end{document}